\documentclass[12pt]{amsart}
\usepackage[utf8]{inputenc}
\usepackage{accents}
\usepackage{amsthm}
\usepackage{amsmath}
\usepackage{stmaryrd}
\usepackage{amssymb} 
\usepackage{comment}
\usepackage{graphicx}
\usepackage{mathrsfs}
\usepackage{fullpage}
\usepackage{mathtools}
\usepackage{colonequals}
\usepackage{hyperref}
\usepackage[all]{xy}
\usepackage[usenames,dvipsnames]{xcolor}
\usepackage[ left = 1.2in, right = 1.2in, 
             top = 1.2in, bottom = 1.2in ]{geometry}
\usepackage{tikz-cd}

\numberwithin{equation}{section} 
\numberwithin{figure}{section} 

\newtheorem{theorem}[equation]{Theorem}
\newtheorem{corollary}[equation]{Corollary}
\newtheorem{lemma}[equation]{Lemma}
\newtheorem{prop}[equation]{Proposition}

\theoremstyle{definition}
\newtheorem{remark}[equation]{Remark}
\newtheorem{example}[equation]{Example}

\theoremstyle{definition}
\newtheorem{definition}[equation]{Definition}

\DeclareMathOperator{\Span}{Span}
\DeclareMathOperator{\codim}{codim}
\DeclareMathOperator{\supp}{supp}

\DeclareMathOperator{\Vol}{Vol}
\DeclareMathOperator{\st}{st}
\DeclareMathOperator{\lk}{lk}
\DeclareMathOperator{\Sym}{Sym}
\DeclareMathOperator{\PP}{PP}
\DeclareMathOperator{\Frac}{Frac}

\newcommand{\R}{\mathbb{R}}

\renewcommand{\v}{\mathbf{v}}

\newcommand{\Z}{{\mathbb{Z}}}

\begin{document}

\title{Anisotropy and the $g$-theorem for simplicial spheres}
\markright{The $g$-theorem for simplicial spheres}
\author{Eric Katz}

\begin{abstract}
We give an exposition of the proof of the lower bound part of the $g$-theorem for simplicial  spheres by Adiprasito, Papadakis, and Petrotou.
\end{abstract}

\maketitle

\noindent

\section{Introduction}

A proof of the lower bound part of the $g$-theorem for simplicial homology spheres over fields of characteristic $2$ was given by Papadakis and Petrotou \cite{PP} following the announcement of a proof by Adiprasito \cite{Adiprasito:beyondpositivity}. This proof was simplified in work of Adiprasito--Papadakis--Petrotou \cite{APP}, and we feel it should be better known. We give a self-contained presentation requiring only basic facts about Cohen--Macaulay and Gorenstein complexes that are well known to the combinatorial commutative algebra community. A reader willing to take those facts on faith should be able to understand the proof.

The $g$-theorem, originally the $g$-conjecture of McMullen \cite{McMullen}, characterized the number of faces of a $d$-dimensional simplicial convex polytope $P$. For background on convex polytopes, we recommend \cite{Grunbaum, Ziegler}. Let $f_i$ be the number of $i$-dimensional faces of $P$ where $f_{-1}=1$ corresponds to the empty face. 
Define the $h$-polynomial 
\[h(t)=\sum_{i=0}^d h_it^i\coloneqq \sum_{i=0}^d f_{i-1}(t-1)^{d-i}.\]
By the Dehn--Sommerville relations, the polynomial $h(t)$ is palindromic, i.e., for all $i$, $h_i=h_{d-i}$. As part of the $g$-conjecture, 
McMullen conjectured that the sequence $\{h_i\}$ is unimodal, that is,
\[h_0\leq h_1\leq \dots \leq h_{\lfloor d/2\rfloor}.\]
This conjecture, called the lower bound conjecture and which we will focus on in this note, was resolved by Stanley \cite{Stanley:number} who applied a deep theorem in algebraic geometry, the Hard Lefschetz theorem for intersection homology to toric varieties. A combinatorial proof by McMullen \cite{McMullen:simple} followed. The theorem was extended to more general convex polytopes by Karu \cite{Karu:HL}. The other parts of the $g$-conjecture (i.e., the upper bound conjecture and sufficiency) were settled by McMullen \cite{McMullen} and Billera--Lee \cite{BilleraLee}

One can ask whether the $g$-theorem generalizes from simplicial convex polytopes to simplicial spheres. Here, instead of focusing on convex polytopes, one considers their boundaries which are simplicial spheres, that is, simplicial complexes homeomorphic spheres. Combinatorially, there are many more simplicial spheres than there are boundaries of polytopes \cite{Kalai:many}, so this generalization is significant. The upper bound conjecture was settled by Stanley for simplicial spheres \cite{Stanley:greenbook,Stanley:UBT} using combinatorial commutative algebra techniques.

For the lower bound conjecture, One can try to imitate Stanley's toric geometry proof by developing combinatorial stand-ins for homology groups. Attached to a $(d-1)$-dimensional simplicial complex $\Sigma$ is the Stanley--Reisner (alt., face) ring $K[\Sigma]$, a $K$-algebra for any field $K$ (see \cite{BrunsHerzog,MS:cca,Stanley:greenbook} for background).  When $\Sigma$ is a simplicial sphere (or more generally a Cohen--Macaulay complex), $K[\Sigma]$ has a linear system of parameters $\ell_1,\dots,\ell_d\in K[\Sigma]^1$, i.e., a regular sequence in the degree $1$ part, $K[\Sigma]^1$ of $K[\Sigma]$. The quotient (often called the Artinian reduction) 
\[A^*(\Sigma)\coloneqq K[\Sigma]/(\ell_1,\dots,\ell_d)\]
is a graded ring  that behaves as if it were the cohomology of a smooth toric variety (with halved grading). Indeed, there is a volume map analogous to the degree map on top-dimensional cohomology,
\[\Vol\colon A^d(\Sigma)\to K,\]
and $A^*(\Sigma)$ obeys Poincar\'{e} duality. Moreover, the ranks of its graded parts correspond to the coefficients of $h(t)$, that is, 
$h_i=\dim A^i(\Sigma)$.

The main result of \cite{APP,PP} that we will discuss here is the lower bound part of the $g$-theorem for simplicial homology spheres
\begin{theorem}
    Let $\Sigma$ be a simplicial homology sphere over a field of characteristic $2$. Then,
    \[h_0\leq h_1\leq\dots\leq h_{\lfloor d/2\rfloor}.\]
\end{theorem}

Recall that for simplicial polytopes, this is a consequence of the Hard Lefschetz theorem which asserts that there is a {\em Lefschetz element} $\ell\in A^1(\Sigma)$, i.e., for any $i$ with $0\leq i\leq d/2$, multiplication by $\ell^{d-2i}$,
\[\ell^{d-2i}\cdot\colon A^i(\Sigma)\to A^{d-i}(\Sigma)\]
is an isomorphism. Because this implies that $\ell\cdot\colon A^i(\Sigma)\to A^{i+1}(\Sigma)$ is injective for $i<d/2$, it gives $h_i\leq h_{i+1}$. An element that satisfies this injectivity property is called a {\em weak Lefschetz element}.
McMullen's proof fits into the framework of Lefschetz theorems: he showed the existence of a (strong) Lefschetz element for convex simplicial polytopes. His proof used convexity in an essential way; this convexity corresponds to a property called positivity in algebraic geometry \cite{Lazarsfeld1}.

Once one has a single Lefschetz element, by a genericity argument, one can conclude that most elements of $A^1(\Sigma)$ are Lefschetz elements. In fact, the non-Lefschetz elements are a Zariski closed set, that is, they are cut out by finitely-many polynomial equations. This suggests that for simplicial spheres, one should hunt for a generic Lefschetz element as no notion of positivity may be possible. Indeed, this is what Adiprasito did in \cite{Adiprasito:beyondpositivity}. There are a series of reduction steps. If one sets $e=\lfloor (d-1)/2\rfloor$, one needs only show
\[\ell\cdot\colon A^e(\Sigma)\to A^{e+1}(\Sigma)\]
is injective. A major insight of \cite{PP} was that it suffices to show that for any nonzero $u\in A^e(\Sigma)$, $u^2\neq 0$. In the case where $d$ is odd, one need only show $\Vol(u^2)\neq 0$. This property is called {\em anisotropy}. There is an analogous condition for $d$ even.

The proof of Papadakis and Petrotou \cite{PP} established anisotropy by means of differential operators in characteristic $2$. Here, the philosophy is informed by polytopes. We describe the situation for odd-dimensional simplicial spheres. The choice of the linear system of parameters $\{\ell_1,\dots,\ell_d\}$ corresponds to a choice of coordinates for the vertices of the simplicial sphere. The volume function $u\mapsto \Vol(u^2)$ is a very complicated expression involving the coordinates. To prove that it does not vanish, one makes the maximally generic choice: the vertex $v$ should be placed at the point $(a_{v1},a_{v2},\dots,a_{vd})\in K^d$ where the set of coordinates $\{a_{vj}\}$ is an algebraically independent set. This forces us to work over the field $K=k(a_{vj})$ where $k$ is some field. To show  the nonvanishing of $\Vol(u^2)$, Papadakis and Petrotou prove the nonvanishing of the derivative of $\Vol(u^2)$ when one moves the vertices. The expression for the derivative $\Vol(u^2)$ is still very complicated, but it becomes much simpler in characteristic $2$. Indeed it is shown not to vanish by some fundamental identities (Proposition~\ref{p:mainidentityodd} and Proposition~\ref{p:mainidentityeven}). This, in turn, guarantees the existence of a weak Lefschetz element and thus a proof of the $g$-theorem. A related fundamental identity (and its strengthening) and its generalization to characteristic $p$ was proved using invariant theory by Karu--Xiao \cite{KX} and Karu--Larson--Stapledon \cite{KLS}, respectively.

We would like to acknowledge the work that informed ours. This note is a report on the paper of Adiprasito--Papadakis--Petrotou \cite{APP}, and we make no claims of originality. 
Our exposition owes much to the work of Karu--Xiao \cite{KX}.
Many of the important geometric ideas in this paper are inspired by the work of Ed Swartz, and we enthusiastically recommend \cite{Swartz}. We would like to thank Karim Adiprasito, Matt Larson and Alan Stapledon for helpful comments about an earlier draft of this manuscript.

\section{Simplicial complexes and Stanley--Reisner rings}
\begin{definition}
A {\em simplicial} complex $\Sigma$ on a finite set $V$ is a non-empty collection of subsets of $V$ that is closed under taking subsets: if $\sigma\in \Sigma$ and $\tau\subseteq \sigma$, then $\tau\in \Sigma$. In particular $\varnothing\in\Sigma$.
\end{definition}

Attached to $\Sigma$ is a topological space $|\Sigma|$, its {\em geometric realization}: let $\Delta_{|V|-1}$ be the standard simplex on $V$, i.e., the subset of $\R^V$ given by 
\[\Delta_{|V|-1}\coloneqq \left\{(t_v)\in\R^V \Big| \sum_v t_v=1\right\};\]
set $|\Sigma|$ to be the subset of $\Delta_{|V|-1}$ given by faces corresponding to elements of $\Sigma$. Thus, to $\sigma\in \Sigma$ corresponds the simplex $|\sigma|$ given by
\[|\sigma|\coloneqq \{(t_v)\in \Delta_{|V|-1}\mid t_v=0 \text{ for } v\not\in\sigma\}.\]
Motivated by the geometric realization, we define $\Sigma_{(k)}$ to be the {\em set of all $k$-dimensional faces of $\Sigma$}, i.e., the set of all $(k+1)$-element members of $\Sigma$. In particular, the empty face $\varnothing$ is $(-1)$-dimensional. We say that $\Sigma$ is purely $(d-1)$-dimensional if all maximal faces are $(d-1)$-dimensional.

For a face $\sigma\in\Sigma$, we define the star and link of $\sigma$ to be
\begin{align*}
\st_{\Sigma}(\sigma)&\coloneqq \{\tau\in\Sigma\mid \sigma\cup\tau\in\Sigma\}\\
\lk_{\Sigma}(\sigma)&\coloneqq \{\tau\in \Sigma\mid \sigma\cup\tau\in \Sigma,\ \sigma\cap\tau=\varnothing\}.
\end{align*}

For a simplicial complex $\Sigma$, and $v$ an element not in $\Sigma_{(0)}$, the {\em cone over $\Sigma$ with apex $v$} is the simplicial complex
\[v\Sigma\coloneqq \Sigma\cup \{\sigma\cup\{v\}\mid \sigma\in \Sigma\}.\]
For $v_+,v_-$, elements not in $\Sigma_{(0)}$, the {\em suspension with poles $v_+$ and $v_-$} is the
simplicial complex
\[S\Sigma\coloneqq \Sigma\cup \{\sigma\cup\{v_+\}\mid \sigma\in \Sigma\}\cup \{\sigma\cup\{v_-\}\mid \sigma\in \Sigma\}.\]

We say a purely $(d-1)$-dimensional simplicial complex is a {\em pseudomanifold without boundary} (henceforth, pseudomanifold) if
every $(d-2)$-face is contained in exactly two $(d-1)$-faces and if it is strongly connected: for any  $(d-1)$-faces $\sigma,\sigma'$, there is a {\em face path}, i.e., a sequence of $(d-1)$-faces $\sigma_0=\sigma,\sigma_1,\dots,\sigma_k=\sigma'$ such that $\sigma_i\cap\sigma_{i+1}$ is a $(d-2)$-face for $i=0,\dots,k-1$. A {\em pseudomanifold} is {\em normal} if the link of every face of codimension at least two is connected. In this case, such a link is  also a normal pseudomanifold.  For a ring $R$, a purely $(d-1)$-dimensional simplicial complex $\Sigma$ is a {\em $R$-homology manifold} if for every nonempty face $\tau\in\Sigma$, $\lk_\Sigma(\tau)$ has the same homology groups over $R$ as a $(d-2-\dim(\tau))$-dimensional sphere. If, in addition, $\Sigma$ has the same homology groups (over $R$) as a $(d-1)$-dimensional sphere, $\Sigma$ is a {\em $R$-homology sphere}. Every connected homology manifold is a normal pseudomanifold.

An {\em orientation on a $(d-1)$-face} $\sigma$ of a $(d-1)$-dimensional pseudomanifold $\Sigma$ is a choice of ordering of elements of $\sigma$ up to even permutation. We say an ordering of elements is {\em positively oriented} if it agrees with the orientation and {\em negative oriented} otherwise. An {\em orientation on the pseudomanifold} $\Sigma$ is a choice of orientation on each $(d-1)$-face such that if $\tau=\{v_0,\dots,v_{d-2}\}\in\Sigma_{d-2}$ is contained in $(d-1)$-faces $\sigma_+=\{v_0,\dots,v_{d-2},v_+\}$ and $\sigma_-=\{v_0,\dots,v_{d-2},v_-\}$, then one of $(v_0,\dots,v_{d-2},v_+)$ and $(v_0,\dots,v_{d-2},v_-)$ is positively oriented and the other negatively oriented.

Attached to $\Sigma$ is its Stanley--Reisner ring. It is the quotient of the polynomial ring on the vertex set $K[x_v\mid v\in V]$ by an ideal $I_\Sigma$. To $S\subset V$, let $x_S\coloneqq \prod_{v\in S} x_v$, and define the ideal $I_{\Sigma}$ by
\[I_\Sigma\coloneqq \left(\left\{x_S\mid S\not\in\Sigma\right\}\right).\]
The {\em Stanley--Reisner} (or face) ring of $\Sigma$ is $K[\Sigma]\coloneqq K[x_v]/I_\Sigma$.
For a nonzero monomial 
\[z=cx_{v_1}^{a_1}\dots x_{v_k}^{a_k}\in K[\Sigma]\]
with $c\in K$ and $a_i\in\Z_{\geq 0}$, we define the support of $z$ to be
\[\supp(z)\coloneqq \{v_i\mid a_i\geq 1\}.\]
We define $\supp(0)\coloneqq \varnothing$. Supports of monomials are necessarily faces of $\Sigma$.
The ring $K[\Sigma]$ is graded by the total degree of an element, and we write $K[\Sigma]^m$ to denote the $K$-vector space spanned by the elements of total degree $m$.

\begin{example} \label{e:polygon1}
  Let $p$ be a positive integer, and let $C_p$ be a $p$-gon, i.e., a cycle graph with $p$ vertices (labeled $v_1,\dots,v_p$ and considered cyclically) and $p$ edges (denoted $v_1v_2,\dots,v_{p-1}v_p,v_pv_1$).
  Then, $K[C_p]=K[x_{v_1},\dots,x_{v_p}]/I_{C_p}$ where
  \[I_{C_p}=\left(x_{v_iv_j}\mid i-j\neq -1,0,1\right)\]
\end{example}

There is functoriality for Stanley--Reisner rings. If $\Sigma'\subseteq \Sigma$ is a subcomplex with the same vertex set $\Sigma'_{(0)}=\Sigma_{(0)}$, then $I_{\Sigma}\subseteq I_{\Sigma'}$, and there is a natural quotient map $\pi_{\Sigma\to \Sigma'}\colon K[\Sigma]\to K[\Sigma']$. This makes $K[\Sigma']$ into a $K[\Sigma]$-module. If $\Sigma''$ is a subcomplex of $\Sigma'$ obtained by deleting some isolated vertices, there is a natural homomorphism $K[\Sigma']\to K[\Sigma'']$ obtained by taking $x_w$ to $0$ for each $w\in \Sigma'_{(0)}\setminus \Sigma''_{(0)}$. For any $v\in \Sigma_{(0)}$, we obtain a homomorphism $K[\Sigma]\to K[\st_\Sigma(v)]$ by composing these two homomorphisms.

If $\Sigma'\subseteq \Sigma$ is an induced subcomplex, i.e., $\sigma\in\Sigma$ is a face of $\Sigma'$ if and only if all the vertices of $\sigma$ belong to $\Sigma'$, then there is a natural homomorphism $i_{\Sigma'\to\Sigma}\colon K[\Sigma']\to K[\Sigma]$. Indeed, the inclusion $\Sigma'_{(0)}\hookrightarrow \Sigma_{(0)}$ induces an inclusion 
\[i\colon K[x_v\mid v\in \Sigma'_{(0)}]\to K[x_v\mid v\in\Sigma_{(0)}]\]
with $i(I_{\Sigma'})\subseteq I_{\Sigma}$. Because $\Sigma$, considered as a subcomplex of its cone $\v\Sigma$ is induced, it gives a ring homomorphism $i_{v}\colon K[\Sigma]\to K[v\Sigma]$.

A regular sequence in a ring $R$ is a sequence of elements $\ell_1,\dots,\ell_k$ such that the following conditions are satisfied: for all $i$, $\ell_i$ is not a zero divisor in $R/(\ell_1,\dots,\ell_{i-1})$, and $R/(\ell_1,\dots,\ell_k)\neq 0$.
For a $(d-1)$-dimensional simplicial complex $\Sigma$, a $d$-dimensional linear subspace $M\subset K[\Sigma]^{1}$ is a {\em regular parameter subspace} if it contains a regular sequence of length $d$. 

For a vertex $v\in\Sigma_{(0)}$, let $e_v\colon K[\Sigma]^{1}\to K$ be given by $\sum a_wx_w\mapsto a_v$, i.e., it takes an element to the coefficient of $v$.
Also, there is a linear isomorphism 
\begin{align*}
  K[\Sigma]^{1} &\cong K^{\Sigma_{(0)}}\\
  \sum_v a_vx_v &\mapsto [v\mapsto a_v].
\end{align*}
A linear subspace $M\subseteq K[\Sigma]^1$ induces a linear inclusion $M\hookrightarrow K^{\Sigma_{(0)}}$ with dual 
\[h\colon K\Sigma_{(0)}=(K^{\Sigma_{(0)}})^\vee\to N\coloneqq M^\vee\] 
where $K\Sigma_{(0)}$ is the vector space of $K$-linear combinations of elements of $\Sigma_{(0)}$. Intuitively, we can view $h$ as a map on the geometric realization $|\Sigma|\to N$. Indeed, one can extend the  map of vertices $h\colon \Sigma_{(0)}\to N$ linearly on faces. To make this concrete, let $\ell_1,\dots,\ell_m$ be a basis for $M$, and write
\[\ell_j=\sum_{v\in\Sigma_{(0)}} a_{vj}x_v.\]
Let $e_1,\dots,e_m$ be the basis of $N$ dual to $\ell_1,\dots,\ell_m$. Then $h\colon K\Sigma_{(0)}\to N $ is obtained by extending the following map linearly:
\[h(v)=\sum_{j=1}^m a_{vj}e_j.\]
 
\begin{remark}
  One can obtain a picture reminiscent of fans defining toric varieties \cite{CLS,Fulton:toricvarieties} by taking the face fan $\Delta_\Sigma$ of $\Sigma$ and linearly extending the map to get $h\colon \Delta_\Sigma \to N$. Of course, in this case, under $h$, the cones of the fan may intersect in their interiors.   
\end{remark}

\begin{definition}
    A linear subspace $M\subset K[\Sigma]^{1}$ is {\em linearly generic} if for any set of vertices $S\subseteq V$ with $|S|\leq \dim M$, the set $h(S)$ spans a $|S|$-dimensional subspace.
\end{definition}

Linear genericity implies that the vertices are mapped into general position: that the preimage under $h$ of each $k$-dimensional linear subspace of $N$ contains at most $k$ points of $V$. Equivalently, for any $S=\{v_1,\dots,v_m\}\subseteq V$ with $m\leq \dim M$,
\[\codim\left(\bigcap_{i=1}^m e_{v_i}^{-1}(0)\subseteq M\right)=m.\]

We define $A_M^*(\Sigma)\coloneqq K[\Sigma]/(M)$, the quotient of $K[\Sigma]$ by the ideal generated by $M$ for a linear subspace $M\subset K[\Sigma]^{1}$. We will suppress $M$ in the notation when it is understood. We will call the graded ring $A_M^*(\Sigma)$ {\em the Chow ring of $\Sigma$ with respect to $M$} in analogy with simplicial toric varieties.

Now, under certain conditions, $A^*_M(\Sigma)$ obeys Poincar\'{e} duality (see, for example, \cite{BrunsHerzog}):

\begin{theorem}
    If $\Sigma$ is a $(d-1)$-dimensional simplicial $K$-homology sphere, then there is a $d$-dimensional regular parameter subspace $M$. In this case,
    \[A^m_M(\Sigma)\cong\begin{cases}
        0 &\text{if }m>d,\\
        K & \text{if }m=d
    \end{cases}\] 
    Fix an isomorphism $\lambda:A^d_M(\Sigma)\to K$.
    For $0\leq m\leq d$, the pairing 
    \begin{align*}
    \langle\ ,\ \rangle\colon A_M^m(\Sigma)\times A_M^{d-m}(\Sigma)&\to K\\
    (x,y)  &\mapsto \lambda(xy)
    \end{align*}
    is non-degenerate, i.e., it induces an isomorphism $A_M^m(\Sigma)\cong A^{d-m}_M(\Sigma)^*$.
\end{theorem}

\section{Displacement Lemma}

Given $x\in A^*(\Sigma)\coloneqq K[\Sigma]/(M)$, we will represent $x$ by an element of  $K[\Sigma]$ given by a sum of square-free monomials with support disjoint from a certain face of $\Sigma$.

\begin{lemma} \label{l:displacement} Let $\Sigma$ be a $(d-1)$-dimensional normal pseudomanifold. Let $M$ be a linearly generic $d$-dimensional subspace of $K[\Sigma]^{1}$. Let $\gamma\in \Sigma_{(d-m-1)}$. Then, any $x\in A^m(\Sigma)$ is equivalent to a linear combination of square-free monomials of the form $\sum_{\eta\in U_\gamma} a_\eta x_\eta$ where $a_\eta\in K$ and
\[U_\gamma\coloneqq \{\eta\in \Sigma_{(m-1)}\mid \eta\cap \gamma=\varnothing\},\]
i.e., that the support of each monomial is disjoint from $\gamma$. 
Moreover, for any $\tau\in (\lk_{\Sigma}(\gamma))_{(m-1)}$, we may choose the sum to be of the form $\sum_{\eta\in U_\gamma^\tau} a_\eta x_\eta$ where
\[U_\gamma^\tau\coloneqq \left(U_\gamma\setminus \lk_{\Sigma}(\gamma)_{(m-1)}\right)\cup\{\tau\},\]
i.e., that $\tau$ is the only $(m-1)$-face in the sum contained in $\lk_{\Sigma}(\gamma)$.
\end{lemma}

\begin{proof}
It suffices to prove this lemma for a monomial $x=\prod_v x_v^{m_v}$ where $m_v\in\Z_{\geq 0}$.
For the monomial $x$, we define a measure of how far it fails to be square-free or have support avoiding $\gamma$:
\[\delta(x)\coloneqq \left(\sum_{v\mid m_v\neq 0} (m_v-1)\right)+|\supp(x)\cap \gamma|.\]
If $\delta(x)=0$, we are done. Otherwise, we will rewrite $x$ as a sum of monomials with smaller $\delta$. Let $v$ be a vertex with $m_v\geq 2$ or with $v\in\supp(x)\cap \gamma$. Hence, $x_v^{-1}x\in K[\Sigma]$.
Now, 
\begin{align*}
|\supp(x)\cup \gamma|&=|\supp(x)|+|\gamma|-|\supp(x)\cap \gamma|\\
&=m-\delta(x)+|\gamma|\\
&=m-\delta(x)+(d-m)\\
&=d-\delta(x)<d\\
\end{align*}
By linear genericity, we may pick $\ell\in M$ such that $e_v(\ell)=1$ and $e_u(\ell)=0$ for $u\in \left(\supp(x)\cup\gamma\right)\setminus\{v\}$. In $A^*(\Sigma)$,
\[0=\ell\cdot (x_v^{-1}x)=x+\sum_{u\not\in \supp(x)\cup\gamma} e_u(\ell)x_u\cdot\left(x_v^{-1}x\right).
\]
The value of $\delta$ on each summand on the right is at most $\delta(x)-1$. We continue this process until we can rewrite $x$ as a sum of square-free monomials corresponding to elements of $U$.

Now, for $\eta\in (\lk_{\Sigma}(\gamma))_{(m-1)}$, we must rewrite $x_\eta$ as a sum of monomials corresponding to faces in $U_\gamma^\tau$. We will accomplish this by means of a face path from $\eta$ to $\tau$ in $\lk_{\Sigma}(\gamma)$, which is strongly connected because $\Sigma$ is a normal pseudomanifold. The repeated step will be passing from one $(m-1)$-dimensional face $\eta_-$ to an adjoining $(m-1)$ face $\eta_+$. 
Let $\eta_+$ and  $\eta_-$ be $(m-1)$-dimensional faces of $\lk_{\Sigma}(\gamma)_{(m-1)}$ sharing a $(m-2)$-dimensional face $\kappa$.
We will write $x_{\eta_-}$ as a linear combination of $x_{\eta_+}$ and monomials corresponding to $(m-1)$-faces not in $\lk_{\Sigma}(\gamma)$. Write $\eta_-=\kappa\cup\{v_-\}$ and $\eta_+=\kappa\cup\{v_+\}$. Pick $\ell\in L$ such that $e_{v_-}(\ell)=1$ and $e_u(\ell)=0$ for $u\in \kappa\cup \gamma$. Then, 
\[0=\ell \cdot x_\kappa = x_{v_-}x_\kappa+\sum_{u\not\in \kappa\cup\gamma\cup \{v_-\}} e_u(\ell)x_ux_\kappa
=x_{\eta_-}+e_{v_+}(\ell)x_{\eta_+}+\sum_{u\not\in \kappa\cup\gamma\cup\{v_-,v_+\}} e_u(\ell)x_u x_\kappa\]
No summand on the right corresponds to a top-dimensional face of $\lk_{\Sigma}(\gamma)$.
\end{proof}

\begin{corollary}
    Every $x\in A^*(\Sigma)$ can be written as the sum of square-free monomials.
\end{corollary}

Also, $A^{m}(\Sigma)\cong 0$ for $m>d$.
If we consider the case where $m=d$ and  $\gamma$ is the empty face, we have the following:

\begin{corollary} \label{l:topdegree}
For any $\tau\in \Sigma_{(d-1)}$, the vector space $A^d(\Sigma)$ is spanned by $x_\tau$.
\end{corollary}

From this, we can conclude that $\dim A^d(\Sigma)\leq 1$.
\section{The volume map}

Following Karu and Xiao \cite{KX}, we will use piecewise polynomials (inspired by \cite{Billera}) to write down an explicit degree map, $\Vol\colon A^d(\Sigma)\to K$ following Brion \cite{Brion}. We call it ``$\Vol$'' to avoid confusion  with the gradings on the relevant rings. This approach is a closely related alternative to the use of stresses by Lee \cite{Lee:stresses}.

Let $\Sigma$ be a $(d-1)$-dimensional pseudomanifold.
Fix once and for all an isomorphism $\wedge^d N\cong K$.
Recall that $M\subset K[\Sigma]^{1}$ induces $h\colon \Sigma_{(0)}\to N$. For $\sigma\in\Sigma$, let $N_\sigma\coloneqq \Span_K(h(\sigma))$, the image of the vector space spanned by the image of $\sigma$ under $h$. Set $M_\sigma\coloneqq N_{\sigma}^\vee=M/h(\sigma)^\perp$. For $\tau\subset \sigma$, there is a natural inclusion $N_\tau\hookrightarrow N_\sigma$ inducing a surjection $M_\sigma\to M_\tau$.

Let $\Sym^m M_\sigma$ denote the $m$th the symmetric power of $M_\sigma$ interpreted as homogeneous degree $m$ polynomials on $N_\sigma$. For $\tau\subseteq \sigma$, there is a linear map $\Sym^m M_\sigma\to \Sym^m M_\tau$, and we will denote the image of $f$ under that map by $f|_{\tau}$ which can be interpreted as the restriction of the function $f$ to $N_\tau$.
Write $\Sym^* M_\sigma\coloneq \bigoplus_m \Sym^m M_\sigma$ for the symmetric algebra on $M_\sigma$. We define piecewise polynomials on $\Sigma$ with respect to $M$ to be
\[\PP_M^*(\Sigma)\coloneqq \{(f_\sigma)_{\sigma\in \Sigma_{(d-1)}}\in (\Sym^* M)^{\Sigma_{(d-1)}}\ {\mid}\ (f_{\sigma_+})|_\tau=(f_{\sigma_-})|_\tau \text{ for all }\tau\subsetneq \sigma_+,\sigma_-\},\]
that is a collection of polynomials $f_\sigma\in \Sym^* M$ indexed by top-dimensional faces that agree on the spans of smaller-dimensional faces. We will suppress ``$M$'' when it is understood. The ring $\PP^*(\Sigma)$ is graded by the degree of polynomials: we will write $\PP^m(\Sigma)$ for the vector space consisting of $(f_{\sigma})$ for which $f_\sigma\in \Sym^m M$ for each $\sigma$.
If one takes the the cone over $\Sigma$ in $N$, these are analogous to the usual piecewise polynomials from toric geometry (see e.g.,~\cite{Payne:equivariant}). 
There is a natural linear map $\Sym^*M\to \PP^*(\Sigma)$ taking $f\mapsto (f_\sigma\coloneqq f)$, and thus we may consider a global polynomial as a piecewise polynomial. This makes $\PP^*(\Sigma)$ into a $\Sym^*M$-algebra. We write $(M)\subset\PP_M^*(\Sigma)$ for the ideal generated by $M$. 

For an oriented $(d-1)$-face $\sigma=(v_1,\dots,v_d)$, defined up to even permutation, we define  $[\sigma]\in \wedge^d N\cong K$ by 
\[[\sigma]\coloneqq h(v_1)\wedge\dots \wedge h(v_d),\]
which is nonzero by linear genericity since $\{h(v_1),\dots,h(v_d)\}$ is a linearly independent set. For any $v_i\in \sigma$, we define a linear map 
\begin{align*}
    [\sigma_v]\colon N&\to \wedge^d N\\
    y&\mapsto h(v_1)\wedge\dots\wedge h(v_{i-1})\wedge y \wedge h(v_{i+1})\wedge \dots \wedge h(v_d).
\end{align*}
We will extend $[\ \ ]$ and $[(\ \ )_v]$ to ordered $d$-tuples of vertices up to alternating permutation, i.e., to elements of $\Sigma_{(0)}^d/A_d$. Set $[(v_1,\dots,v_d)]=h(v_1)\wedge\dots\wedge h(v_d)$ and similarly for $[(v_1,\dots,v_d)_{v_i}]$. For $T=(v_1,\dots,v_d)$ and $w\in \Sigma_{(0)}$, write 
\[T_{v_i}^w\coloneqq (v_1,\dots,v_{i-1},w,v_{i+1},\dots,v_d),\]
i.e., we substitute $w$ for $v_i$.

For $v\in\Sigma_{(0)}$, we define the Courant function $\Psi_v\in \PP^1(\Sigma)$ by
\[(\Psi_v)_{\sigma}\coloneqq
\begin{cases}
    [\sigma_v]/[\sigma] &\text{if }v\in\sigma\\
    0 &\text{else}.
\end{cases}
\]
This function is characterized by the fact that it is equal to $1$ on $v$, $0$ on any other vertex, and its restriction to any top-dimensional face is linear. It is a straightforward verification that it is indeed a piecewise-linear function.
For $\tau\in\Sigma$, then $\Psi_\tau\coloneqq \prod_{v\in\tau} \Psi_v$ is an element of $\PP^{|\tau|}(\Sigma)$.

We define for a top-dimensional face $\sigma={v_1,\dots,v_d}$,
\[\chi_\sigma\coloneq \frac{[\sigma_{v_1}]}{[\sigma]}\dots\frac{[\sigma_{v_d}]}{[\sigma]}\in \Sym^d M.\]
For example, if $\sigma=(e_1,\dots,e_d)$ corresponds to the the standard orthant in $K^d$, then $\chi_{\sigma}=x_1\dots x_d$ where $x_1,\dots,x_d$ are coordinates on $K^d$.

\begin{lemma} The association $x_v\mapsto \Psi_v$ induces a $(\Sym^*M)$-algebra isomorphism $q\colon K[\Sigma]\to \PP^*(\Sigma)$.
\end{lemma}

The proof is a straightforward verification. For completeness, we explain the $(\Sym^* M)$-action. 
For $m\in M$, in $K[\Sigma]$, we write $m=\sum_v e_v(m)x_v\in K[\Sigma]$. This maps to
$f_m=\sum_v e_v(m)\Psi_v$, a piecewise-linear function characterized by the property that on each vertex $v$, it takes the value $e_v(m)$. Because for each top-dimensional 
cone $\sigma\in\Delta$, $(f_m)_{\sigma}$ agrees with $m$ on the vertices of $\sigma$, we have $(f_m)_{\sigma}=m$. Alternatively, if we identify a face $\sigma$ with its image in $N$ under $h$, the function $f_m$ is equal to $m\in M=N^\vee$ on $\sigma$.

\begin{definition} Let $\Sigma$ be an oriented $(d-1)$-dimensional pseudomanifold.
    The degree map $\Vol\colon \PP_M^*(\Sigma)\to\Frac(\Sym^*M)\otimes (\wedge^d M)^\vee$ is given by
    \[f\mapsto \sum_{\sigma\in\Sigma_{(d-1)}} \frac{f_{\sigma}}{\chi_\sigma[\sigma]}\]
    where $\sigma$ is given its orientation in $\Sigma$.
\end{definition}

This will turn out to to vanish on the ideal $(M)$ and to be valued in $\Sym^*M\otimes (\wedge^d M)^\vee$. Note that $\Vol$ is $(\Sym^*M)$-linear (where $\Sym^*M$ acts trivially on $(\wedge^d M)^\vee$). We will write $\Vol$ for the composition $\Vol\circ\ q$. For $m\leq -1$, we follow the convention that $\Sym^m M=\{0\}$.

\begin{lemma} For any $f\in \PP_M^m(\Sigma)$, $\Vol(f)\in (\Sym^{m-d} M)\otimes (\wedge^d M)^\vee.$
\end{lemma}

\begin{proof}
    It suffices to show $\Vol(f)\in (\Sym^* M)\otimes (\wedge^d M)^\vee$, that is, $\Vol(f)$ is a polynomial. A priori, it is a rational function, but its only possible poles are simple poles on  $N_\tau$ for $\tau\in\Sigma_{(d-2)}$; indeed, $[\sigma_v]$'s only zero is a simple one on  $N_\tau$ for $\tau=\sigma\setminus v$. We'll show that $\Vol(f)$ is, indeed, regular on $N_\tau$. Suppose that $\sigma_+,\sigma_-\in \Sigma_{(d-1)}$ are the only faces containing  $\tau$, so $\sigma_+=\tau\cup \{v_+\}$ and $\sigma_-=\tau\cup \{v_-\}$ and $[(\sigma_+)_{v_+}]=-[(\sigma_-)_{v_-}]$ where the sign comes from the orientation.
    The vanishing of the pole on $N_\tau$ is equivalent to the vanishing of 
    \begin{align*}
    \left([(\sigma_+)_{v_+}]\left(\frac{f_{\sigma_+}}{\chi_{\sigma_+}[\sigma_+]}+\frac{f_{\sigma_-}}{\chi_{\sigma_-}[\sigma_-]}\right)\right)\Bigg|_{\tau}&
    =\left([(\sigma_+)_{v_+}]\left(\frac{f_{\sigma_+}}{[(\sigma_+)_{v_+}](\Psi_\tau)_{\sigma_+}}
    +\frac{f_{\sigma_-}}{[(\sigma_-)_{v_-}](\Psi_\tau)_{\sigma_-}}\right)\right)\Bigg|_\tau\\
    &=\left(\frac{f_{\sigma_+}}{(\Psi_\tau)|_{\sigma_+}}
    -\frac{f_{\sigma_-}}{(\Psi_\tau)|_{\sigma_-}}\right)\Bigg|_\tau=0. \qedhere
  \end{align*}
\end{proof}

\begin{corollary}
    For any $f\in \PP_M^d(\Sigma)$, if $f\in (M)$, then $\Vol(f)=0$.
\end{corollary}

\begin{proof}
    It suffices to show this for $f=mg$ for $m\in M$ and $g\in\PP_M^{d-1}(M)$. Since $\Vol(g)\in \Sym^{-1}M\otimes (\wedge^d M)^\vee=\{0\}$, $\Vol(f)=m\Vol(g)=0.$
\end{proof}

For $\sigma\in\Sigma_{(d-1)}$, $\Vol(x_\sigma)=[\sigma]^{-1}$. Since we know $A^d(\Sigma)$ is at most $1$-dimensional and possesses a nonzero map to $K$, we see that $A^d(\Sigma)$ is $1$-dimensional. 

\begin{example} \label{e:polygon2}
    Consider $C_p$ from Example~\ref{e:polygon1}, oriented so that for all $j$, $(v_j,v_{j+1})$ is positively oriented. Let $M\subseteq K[C_p]^1$ be a $2$-dimensional linearly generic subspace spanned by $\ell_1,\ell_2$ where
    \[\ell_i=\sum_{j=1}^p a_{v_ji}x_{v_j}.\]
    Then, 
    \[\Vol(x_{v_j}x_{v_k})=\begin{cases}
        0 &\text{if } k-j\neq -1,0,1,\\
        \pm[v_jv_k]^{-1} &\text{if }k-j=\pm 1
    \end{cases}\]
    where the sign is determined by whether $v_jv_k$ is positively or negatively oriented.  
    Note that 
    \[[v_jv_k]=
    \begin{vmatrix}
      a_{{v_j}1}&a_{{v_k}1}\\
      a_{{v_j}2}&a_{{v_k}2}
    \end{vmatrix}.\]
    under the isomorphism between $\wedge^2 M$ and $K$.
    Now, 
    \begin{align*}
        \Vol(x_{v_j}^2)&=\Vol\left(x_{v_j}\left(-a_{v_j1}^{-1}\sum_{k\neq j}a_{v_k1} x_{v_i}\right)\right)\\
        &=-a_{v_j1}^{-1}\Vol(a_{v_{j-1}1}x_{v_{j-1}}x_{v_j}+a_{v_{j+1}1}x_{v_j}x_{v_{j+1}})\\
        &=-a_{v_j1}^{-1}a_{v_{j-1}1}[v_{j-1}v_j]^{-1}-a_{v_j1}^{-1}a_{v_{j+1}1}[v_jv_{j+1}]^{-1}.
    \end{align*}
\end{example}

\section{The cone and star lemmas}

We include some lemmas on the Chow rings of cones and stars. The philosophy behind these lemmas owes much to the work of Ed Swartz (see, e.g.,~\cite{Swartz}).

\begin{lemma}[Cone Lemma] \label{l:conelemma} Let $\Sigma$ be a simplicial complex with cone $v\Sigma$ giving a homomorphism $i_v\colon K[\Sigma]\to K[v\Sigma]$. Let $M_{v\Sigma}\subset K[v\Sigma]^1$ be a subspace such that $e_v\colon M_{v\Sigma}\to K$ is surjective, and set $M_{\Sigma}=i_v^{-1}(M_{v\Sigma})$. Then, $i_v$ induces an isomorphism $i_{v}\colon A_{M_\Sigma}^*(\Sigma)\cong A_{M_{v\Sigma}}^*(v\Sigma).$
\end{lemma}

\begin{proof}
  The existence of the homomorphism $i_v\colon A^*(\Sigma)\to A^*(v\Sigma)$ is an immediate consequence of $i(M_\Sigma)\subseteq M_{v\Sigma}$. We construct an inverse $j_v$ to $i_v$.
  Pick $f\in M_{v\Sigma}$ with $e_v(f)=1$. Write
    \[f=x_v+\sum_{w\in \Sigma_{(0)}} a_wx_w.\] 
  The inverse $j_v$ is defined by $j_v(x_w)\coloneqq x_w$ for $w\neq v$ and 
  \[j_v(x_v)\coloneqq -\sum_{w\in \Sigma_{(0)}} a_wx_w.\]
  It is straightforward to verify that both $i_v\circ j_v$ and $j_v\circ i_v$ are identity maps.
\end{proof}

The next two results concern the case where $v$ is a vertex of $\Sigma$, a $(d-1)$-dimensional normal pseudomanifold. Let $M_\Sigma\subset K[\Sigma]^1$ be a linear subspace. Then, there is a natural homomorphism $\pi_v\colon K[\Sigma]\to K[\st_{\Sigma}(v)]$ given by mapping $x_w\mapsto 0$ for $w\not\in \st_{\Sigma}(v)_{(0)}$. Set $M_{\st_{\Sigma}(v)}\coloneqq \pi_v(M_\Sigma)$. We can identify $\st_\Sigma(v)$ with $v\lk_\Sigma(v)$, and let $M_{\lk_{\Sigma}(v)}$ be induced from $M_{\st_{\Sigma}(v)}$ as in Lemma~\ref{l:conelemma}. 
The following is immediate:

\begin{corollary}
   Under the above assumptions,  $A^*(\st_{\Sigma}(v)\cong A^*(\lk_{\Sigma}(v))$.
\end{corollary}

As can be checked, multiplication by $x_v$,
\[\iota\colon K[\st_{\Sigma}(v)]\to K[\Sigma],\quad u\mapsto x_vu\]
induces a homomorphism of $A^*(\Sigma)$ modules $\iota\colon A^*(\st_{\Sigma}(v))\to A^*(\Sigma).$

\begin{lemma}[Star Lemma, compare Prop.~4.24 of \cite{Swartz}] \label{l:starlemma} Let $\Sigma$ be a $(d-1)$-dimensional normal pseudomanifold with regular parameter space $M_{\Sigma}$. Let $v\in \Sigma_{(0)}$.
Suppose that $\lk_{\Sigma}(v)$ is a $K$-homology sphere and that
$M_{\lk_{\Sigma}(v)}$ is a regular parameter space for $\lk_{\Sigma}(v)$.
For all $m\leq d-1$, $\iota\colon A^m(\st_{\Sigma}(v))\to A^{m+1}(\Sigma)$ is injective.
\end{lemma}

\begin{proof}
  We first prove the injectivity for $m=d-1$. Since $\lk_{\Sigma}(v)$ is a normal pseudomanifold, by Corollary~\ref{l:topdegree}, $A^m(\lk_{\Sigma}(v))\cong A^m(\st_{\Sigma}(v))$ is generated by $x_\tau$ for any $(d-2)$-face $\tau$. Then $\iota(x_\tau)=x_vx_\tau=x_\sigma$ where $\sigma=\tau\cup \{v\}$. Because $\sigma$ is a $(d-1)$-dimensional face of $\lk_{\Sigma}(v)$, $x_\sigma$ is nonzero in $A^d(\Sigma)$.

  It suffices to prove this for homogeneous $a\in A^m(\lk_{\Sigma}(v))$ for all $m$. Since $A^*(\lk_{\Sigma}(v))$ obeys Poincar\'{e} duality, there exists $b\in A^{d-1-m}(\lk_{\Sigma}(v))$ such that $ab\neq 0$. We can represent $b$ as an element of $K[\Sigma]$, i.e., there is an element $b'\in K[\Sigma]$ such that $b=b'\cdot 1$ using the $K[\Sigma]$-module structure of $A^*(\lk_{\Sigma}(v))$. By the $K[\Sigma]$-linearity of $\iota$, we see
  $0\neq \iota(ab)=b'\cdot \iota(a)$.
  Hence $\iota(a)\neq 0$.
\end{proof}

\section{Generic parameter space and the main identity}

We now consider the case of the generic parameter space. Let $k$ be a field, and let $(a_{vi})$ be indeterminates indexed by $v\in \Sigma_{(0)}$ and $i\in \{1,\dots,d\}$. Let $K=k(a_{vi})$ be the field of rational functions in $a_{vi}$. Set $\ell_i=\sum_v a_{vi}x_v\in K[\Sigma]$ for $i\in\{1,\dots,d\}$. 
Let $M\coloneqq \Span(\ell_1,\dots,\ell_d)$. 
Use $\{\ell_1,\dots,\ell_d\}$ as a basis for $M$. Let $\{e_1,\dots,e_d\}\subset N=M^\vee$ be the dual basis, so we can write $h(v)=\sum a_{vi}e_i$.

For the above choice, $M$ is a regular parameter space. Indeed, one can verify that $\ell_1,\dots,\ell_d$ is a regular sequence by the discussion in \cite[1.5.10--1.5.12]{BrunsHerzog}. Because each $\ell_i$ is algebraically independent over $K_{i-1}\coloneqq k(\{a_{vj}\mid {j\leq i-1}\})$, it avoids the associated primes of $K_{i-1}[\Sigma]/(\ell_1,\dots,\ell_{i-1})$, thus of $K[\Sigma]/(\ell_1,\dots,\ell_{i-1})$. Hence, $\ell_i$ is not a zero divisor.

For distinct vertices $v\neq w$, we define a differential operator $\partial_v^w\colon K\to K$ by 
\[\partial_v^w\coloneqq \sum_i a_{wi}\frac{\partial}{\partial a_{vi}}.\]
It has the property that $\partial_v^w(k)=0$ and $\partial_v^w(fg)=\partial_v^w(f)g+f\partial_v^w(g)$.
Note that $\partial_v^w h(v)=h(w)$, and thus
\[\partial_v^w[\sigma]=
\begin{cases}
[\sigma_v^w] &\text{if }v\in \sigma,\ w\not\in\sigma\\
0 &\text{else}
\end{cases},\quad 
\partial_v^w[\sigma_u]=
\begin{cases}
[(\sigma_v^w)_u] &\text{if }v\in \sigma\setminus u,\ w\not\in\sigma\setminus \{u,v\}\\
\pm [\sigma_v] &\text{if }v\in \sigma\setminus u,\ w=u\\
0 &\text{else}.
\end{cases}\ 
\]

Given two ordered $k$-tuples of vertices of the same length, $T=(v_1,\dots,v_k)$ and $U=(w_1,\dots,w_k)$, we define
\[\partial_T^U=\partial_{v_1}^{w_1}\partial_{v_2}^{w_2}\dots\partial_{v_k}^{w_k}.\]
We give a proof attributed in \cite{APP} to Geva Yashe of an important identity on odd-dimensional pseudomanifolds.

\begin{prop} \label{p:mainidentityodd} Let $\Sigma$ be a $(d-1)$-dimensional oriented pseudomanifold for $d$ even. Let $\gamma,\tau\in \Sigma_{(d/2-1)}$ be disjoint faces contained in some $\sigma\in\Sigma_{(d-1)}$. Then for any ordering on the vertices of $\gamma$ and $\tau$,
\[\partial_\gamma^\tau \Vol(x_\tau^2)=\pm[\sigma]\Vol(x_\tau x_\gamma)^2.\]
For $\eta\in \Sigma_{(d/2-1)}$ not in $\st_{\Sigma}(\gamma)$, $\partial_\gamma^\tau\Vol(x_\eta^2)=0$.
\end{prop}

With a bit more care, we could write the signs, but they are not important for our purposes.

\begin{proof}
  Write $\gamma=(v_1,\dots,v_{d/2})$ and $\tau=(w_1,\dots,w_{d/2})$
   Because $(\Psi_v)_{\sigma'}=0$ unless $v\in\sigma'$ and thus $(\Psi_\tau)_{\sigma'}=0$ unless $\tau\subset\sigma'$, we have the identity 
  \[\Vol(x_\tau^2)=\sum_{\sigma'\in\Sigma_{(d-1)}} \frac{(\Psi_\tau)^2_{\sigma'}}{\chi_{\sigma'}[\sigma']}
    =\sum_{\substack{\sigma'\in\Sigma_{(d-1)}\\\sigma'\supset\tau}} \frac{(\Psi_\tau)^2_{\sigma'}}{(\Psi_{\sigma'})_{\sigma'}[\sigma']}\\
    =\sum_{\substack{\sigma'\in\Sigma_{(d-1)}\\\sigma'\supset\tau}} \frac{(\Psi_\tau)_{\sigma'}}{(\Psi_{\sigma'\setminus\tau})_{\sigma'}[\sigma']}.\]
  One observes that for $\sigma'\supset \tau$, $\partial_{v_j}^{w_j}[\sigma']=0$: if $v_j\not\in\sigma'$, $[\sigma']$ does not involve any $a_{v_j i}$; if $v_j\in\sigma'$, then  $\partial_{v_j}^{w_j}[\sigma']$
  is a wedge with two copies of $w_j$ and so is zero. Also, $\partial_{v_j}^{w_j}[\sigma'_u]=0$ unless $v_j\in \sigma'\setminus u$. Therefore, the only terms that contribute to $\partial_{\gamma}^\tau\left(\Vol(x_\tau)^2\right)$ are those for which $\sigma'\supset\gamma$, i.e., those for which $\sigma'=\sigma$.
  Hence,
  \[\partial_\gamma^\tau\Vol(x_\tau^2)=\partial_\gamma^\tau\left(\frac{(\Psi_\tau)_{\sigma}}{(\Psi_\gamma)_{\sigma}[\sigma]}\right)=\partial_\gamma^\tau\left(\frac{\prod_{w\in\tau}[\sigma_w]}{[\sigma]\prod_{v\in\gamma}[\sigma_v]}\right)=\pm\frac{\prod_{v\in\gamma}[\sigma_v]}{[\sigma]\prod_{v\in\gamma}[\sigma_v]}=\pm[\sigma]^{-1}.\]
  For the second to last equality, we used the observation that 
  \[\partial_{v_j}^{w_j}[\sigma_{w_k}]=
  \begin{cases}
      \pm[\sigma_{v_j}] &\text{if }j=k\\
      0 &\text{else.}
   \end{cases}
  \]
  We conclude by noting 
  \[[\sigma]^{-1}=\Vol(x_\sigma)=[\sigma]\Vol(x_\sigma)^2=[\sigma]\Vol(x_\gamma x_\tau)^2.\]
  For the second equality, observe
      \[\partial_\gamma^\tau\Vol(x_\eta^2)=\partial_\gamma^\tau\left(\sum_{\sigma'\in\Sigma_{(d-1)}} \frac{(\Psi_\eta^2)_{\sigma'}}{\chi_{\sigma'}[\sigma']}\right).\]
  The only summands which contribute are those for which $\sigma'\supset \gamma$. For those $\sigma'$, $\Psi_\eta|_{\sigma'}=0$.
\end{proof}

\begin{example} \label{e:polygon3}
    Returning to Example~\ref{e:polygon2}, we have
    \[\partial_{v_{j-1}}^{v_j}\Vol(x_{v_j}^2)=-a_{v_j1}^{-1}a_{v_j1}[v_{j-1}v_j]^{-1}=-[v_{j-1}v_j]\Vol(x_{v_{j-1}}x_{v_j})^2.\]
    Similarly, $\partial_{v_{j+1}}^{v_j}\Vol(x_{v_j}^2)=-[v_{j}v_{j+1}]\Vol(x_{v_{j}}x_{v_{j+1}})^2$.
\end{example}

An identical argument gives the analogous results in the even dimensional case:
\begin{prop} \label{p:mainidentityeven} Let $\Sigma$ be a $(d-1)$-dimensional oriented pseudomanifold for $d$ odd. Let $\gamma,\tau\in \Sigma_{((d-1)/2-1)}$ be disjoint faces of $\sigma\in\Sigma_{(d-1)}$. 
Let $p$ be the unique element of $\sigma\setminus (\gamma\cup \tau)$.
Then, for any ordering of the vertices of $\gamma$ and $\tau$,
\[\partial_\gamma^\tau \Vol(x_px_\tau^2)=\pm[\sigma]\Vol(x_px_\tau x_\gamma)^2.\]

For $\eta\in \Sigma_{((d-1)/2-1)}$ disjoint from $\st_{\Sigma}(\gamma\cup\{p\})$, $\partial_\gamma^\tau\Vol(x_px_\eta^2)=0$.
\end{prop}

\section{Anisotropy}

Now, we prove anisotropy when $k$ is a field of characteristic $2$, so we can ignore signs. Let $K=k(a_{vi})$, and take the generic regular parameter space as in the previous section. We strengthen Proposition~\ref{p:mainidentityodd} by replacing $x_\tau$ by an arbitrary $u\in A^{d/2}(\Sigma)$. We will make critical use of the observation that for any $\lambda\in K$, $\partial_{u_i}^{v_i}\lambda^2=0$.

\begin{prop}  Let $\Sigma$ be a $(d-1)$-dimensional oriented pseudomanifold for $d$ even. Let $\gamma,\tau\in \Sigma_{(d/2-1)}$ be disjoint faces contained in $\sigma\in\Sigma_{(d-1)}$. Then, for any ordering of the vertices of $\gamma$ and $\tau$, and for any $a\in A^{d/2}(\Sigma)$,
\[\partial_\gamma^\tau \Vol(a^2)=[\sigma]\Vol(a x_\gamma)^2.\]
\end{prop}

\begin{proof}
    By Lemma~\ref{l:displacement} with $m=\frac{d}{2}$, we may suppose $a=\sum_{\eta\in U_\gamma^\tau}\lambda_\eta x_\eta$. Then, by Proposition~\ref{p:mainidentityodd},
    \begin{multline*}
      \partial_\gamma^\tau\Vol(a^2)=\partial_\gamma^\tau\Vol\left(\sum_{\eta\in U_\gamma^\tau}\lambda_\eta^2 x_\eta^2\right)=
      \left(\sum_{\eta\in U_\gamma^\tau}\lambda_\eta^2 \partial_\gamma^\tau\Vol(x_\eta^2)\right)\\
      =\lambda_\tau^2\partial_{\gamma}^\tau\Vol(x_\tau^2)
      =[\sigma]\Vol(\lambda_\tau x_\tau x_\gamma)^2
      =[\sigma]\Vol\left(\sum_{\eta\in U_\gamma^\tau}\lambda_\eta x_\eta x_\gamma\right)^2=[\sigma]\Vol(ax_\gamma)^2.\qedhere
    \end{multline*}
\end{proof}

\begin{corollary} \label{c:oddisotropy}
    Let $\Sigma$ be a $(d-1)$-dimensional $K$-homology sphere with $d$ even. Let $a\in A^{d/2}(\Sigma)$ be nonzero. Then,
  $a^2\neq 0$.
\end{corollary}

\begin{proof}
    By Poincar\'{e} duality, there exists $b\in A^{d/2}(\Sigma)$ such that $ab\neq 0$. By Lemma~\ref{l:displacement}, we can write $b=\sum_{\eta\in \Sigma_{(d/2-1)}} a_\eta x_\eta$.  Hence there exists $\gamma\in \Sigma_{(d/2-1)}$ such that $ax_\gamma\neq 0$. Pick $\sigma\in \Sigma_{(d-1)}$ containing $\gamma$. Let $\tau\coloneqq\sigma\setminus\gamma$. Then,
    $\partial_\gamma^\tau \Vol(a^2)=[\sigma]\Vol(ax_\gamma)^2\neq 0.$
\end{proof}

Now, we consider the even-dimensional case.

\begin{prop}  Let $\Sigma$ be a $(d-1)$-dimensional oriented pseudomanifold for $d$ odd. Let $\gamma,\tau\in \Sigma_{((d-1)/2-1)}$ be disjoint faces contained in $\sigma\in\Sigma_{(d-1)}$. 
Let $p$ be the unique element of $\sigma\setminus (\gamma\cup \tau)$.
Then for any ordering on the vertices of $\gamma$ and $\tau$, and for any $a\in A^{(d-1)/2}(\Sigma)$,
\[\partial_\gamma^\tau \Vol(x_pa^2)=[\sigma]\Vol(a x_p x_\gamma)^2.\]
\end{prop}

\begin{proof}
  By Lemma~\ref{l:displacement} with $m=\frac{d-1}{2}$, we may suppose $a=\sum_{\eta\in U_{\gamma\cup\{p\}}^\tau}\lambda_\eta x_\eta$. Then, the proof proceeds as in Proposition~\ref{p:mainidentityeven}.   
\end{proof}

\begin{corollary} \label{c:evenisoptropy}
    Let $\Sigma$ be a $(d-1)$-dimensional $K$-homology sphere with $d$ odd. Let $a\in A^{(d-1)/2}$ be nonzero. Then,
  $a^2\neq 0$.
\end{corollary}

\begin{proof}
    Find $b\in A^{(d+1)/2}$ such that $ab\neq 0$. Hence, there exists $\eta\in \Sigma_{((d+1)/2-1)}$ with $ax_\eta\neq 0$. Pick $p\in\eta,$ and set $\gamma=\eta\setminus\{p\}$. Pick $\sigma\in\Sigma_{(d-1)}$ containing $\eta$, and set $\tau=\sigma\setminus\eta$.
    Then, 
    \[\partial_\gamma^\tau\Vol(x_p a^2)=\Vol(a x_p x_\gamma)^2=\Vol(ax_\eta)^2\neq 0. \qedhere \]
\end{proof}

\begin{remark}
  The differential operator $\partial_{\gamma}^{\tau}$ is the composition of operators $\partial_{v_i}^{w_i}$. These can be thought of generators of vector fields on the realization space of $\Sigma$ over $K$, i.e., the space of all maps $h\colon \Sigma^{(0)}\to K^d$. In this case, $\partial_{v_i}^{w_i}$ moves the vertex $v_i$ in the direction of $w_i$. Simplicial complexes have big realization spaces; more general polyhedral complexes, whose combinatorics require vertices lie in particular linear subspaces, have smaller realization spaces, accordingly with fewer vector fields. It would be interesting to understand how the presence or lack of anisotropy constrains realization spaces.
\end{remark}

\section{Proof of weak Lefschetz}

Let $k$ be a field of characteristic $2$. We work with a generic linear parameter space over a purely transcendental extension of $k$. Our main result is the following Weak Lefschetz Theorem.

\begin{theorem} \label{t:weaklefschetz}
    Let $\Sigma$ be a $(d-1)$-dimensional $k$-homology sphere.  Then, there is purely transcendental extension $K$ of $k$ and an element $\ell\in A^1(\Sigma)$ such that for $m<(d-1)/2$, the multiplication map $\ell\cdot \colon A^m(\Sigma)\to A^{m+1}(\Sigma)$ is injective.
\end{theorem}

Note that the hypotheses imply that $\Sigma$ is a $K$-homology sphere.

\begin{lemma}
    To prove Theorem~\ref{t:weaklefschetz} , it suffices to prove the injectivity of $\ell\cdot$ for $m=\lfloor (d-1)/2\rfloor$.
\end{lemma}

\begin{proof}
  Write $e=\lfloor (d-1)/2\rfloor$. If $m=e$, we are done. Otherwise, let $a\in A^m(\Sigma)$ for $m<e$. By Poincar\'{e} duality, there exists $b\in A^{d-m}(\Sigma)$ such that $ab\neq 0$. Hence there is $\eta\in \Sigma_{(d-m-1)}$ such that $ax_\eta\neq 0$. Let $\gamma$ be an $(e-m-1)$-face of $\eta$. Then, $ax_\gamma\neq 0$. Since $ax_\gamma\in A^e(\Sigma)$, 
  \[(\ell a)x_\gamma=\ell(ax_\gamma)\neq 0.\]
  Therefore, $\ell a\neq 0.$
\end{proof}

Let $S\Sigma$ be the suspension of $\Sigma$ with poles $v_+$ and $v_-$. Then, $S\Sigma$ is a simplicial homology sphere, and $\st_{S\Sigma}(v_+)=v_+\Sigma$. Let $M_{S\Sigma}$ be the generic parameter space on $S\Sigma$ over 
\[K\coloneqq k\left(\{a_{vi}\mid v\in S\Sigma_{(0)},\ i\in \{0,\dots,d\}\}\right).\]
Let $M_{\st_{S\Sigma}(v_+)}=\pi_{v_+}(M_{S\Sigma})\subseteq K[\st_{S\Sigma}(v_+)]^1$. 
Set $M_{\lk_{S\Sigma}(v_+)}=i_{v_+}^{-1}(M_{\st_{S\Sigma}(v_+)})$

We note that $\lk_{S\Sigma}(v_+)=\Sigma$, and we set $M_\Sigma\coloneqq M_{\lk_{S\Sigma}(v_+)}$ under the isomorphism $K[\lk_{S\Sigma}(v_+)]\cong K[\Sigma]$.
Now, the isomorphism 
$A^*(\st_{S\Sigma}(v_+))\to A^*(\Sigma)$ takes $x_{v_+}$ to some element $\ell\in A^1(\Sigma)$. This will be our weak Lefschetz element.

Let us write down the parameter space $M_{\Sigma}$. Note that $M_{S\Sigma}$ is spanned by $\tilde{\ell}_0,\dots,\tilde{\ell}_{d}$ where
\[\tilde{\ell}_j\coloneqq a_{v_+j}x_{v_+}+a_{v_-j}x_{v_-}+\sum_{v\in\Sigma_{(0)}} a_{vj}x_v.\]
Hence, $M_{\st_{S\Sigma}(v_+)}$ is spanned by $\{\hat{\ell}_0,\dots,\hat{\ell}_{d}\}$ where
\[\hat{\ell}_j\coloneqq a_{v_+j}x_{v_+}+\sum_{v\in\Sigma_{(0)}} a_{vj}x_v.\]
Therefore, $M_{\Sigma}=M_{\lk_{S\Sigma}(v_+)}$ is spanned by $\{\ell_1,\dots,\ell_d\}$ where
\[\ell_j\coloneqq a_{v_+0}\hat{\ell}_j-a_{v_+j}\hat{\ell}_0=\sum_{v\in\Sigma_{(0)}} (a_{v_+0}a_{vj}-a_{v_+j}a_{v0})x_v.
\]
The coefficients $c_{vj}\coloneqq a_{v_+0}a_{vj}-a_{v_+j}a_{v0}$ are algebraically independent over $k$ because 
\[\{a_{vj} \mid v\in\Sigma_{(0)},1\leq j\leq d\}\] 
is algebraically independent. Hence, $M_\Sigma$ is a regular parameter space over $K$.
Under these homomorphisms, $x_{v_+}\in A^1(\st_{S\Sigma}(v_+))$ maps to $\ell\in A^1(\Sigma)$ given by
\[\ell=-a_{v_+0}^{-1}\sum_{v\in \Sigma_{(0)}} a_{v0}x_v.\]
The injectivity of 
\[\ell\cdot \colon A^e(\Sigma)\to A^{e+1}(\Sigma)\]
is immediate from the following:

\begin{prop} \label{p:weaklefschetzsuspension} For $e=\lfloor (d-1)/2\rfloor$, multiplication map,
\[x_{v_+}\cdot\colon A^e(\st_{S\Sigma}(v_+))\to A^{e+1}(\st_{S\Sigma}(v_+))\] 
is injective.
\end{prop}

\begin{proof}
    Let $a\in A^e(\st_{S\Sigma}(v_+))$ with $x_{v_+}a=0.$ Using the isomorphisms 
    \[A^e(\st_{S\Sigma}(v_+))\cong A^e(\lk_{S\Sigma}(v_+))=A^e(\Sigma),\]
    we may suppose that $a$ is the image of $\tilde{a}\in K[\Sigma]^e$ under the quotient $K[\Sigma]\to A^*(\Sigma)$. We will identify $\tilde{a}$ with its image in $K[S\Sigma]$ under the homomorphism $i\colon K[\Sigma]\to K[S\Sigma]$ (which exists because $\Sigma$ is an induced subcomplex of $S\Sigma$).

  Since $M_{\Sigma}\subset K[\Sigma]^1$ is a regular parameter space, by Lemma~\ref{l:starlemma}, multiplication by $x_{v_+}$ induces an injective $K[S\Sigma]$-homomorphism, 
  \[\iota\colon A^e(\st_{S\Sigma}(v_+))\to A^{e+1}(S\Sigma).\]
  Now, 
  \[\iota(a)^2=x_{v_+}^2\tilde{a}^2=\iota(x_{v_+}a^2)=\iota(0)=0.\]
  By Corollary~\ref{c:oddisotropy} or Corollary~\ref{c:evenisoptropy}, $\iota(a)=0$. By the injectivity of $\iota$, $a=0$. Therefore, multiplication by $x_{v_+}$ is injective.
\end{proof}
    
\begin{remark}
Note that multiplication by $x_{v_+}$ is given by a sort of push-pull,
\[A^*(\st_{S\Sigma}(v_+))\to A^{*+1}(S\Sigma)\to A^{*+1}(\st_{S\Sigma}(v_+)),\]
reminiscent of Deligne's and originally, Lefschetz's arguments for proving the Hard Lefschetz theorem (see \cite{Deligne:Weil2,Lamotke}). It would be interesting to find combinatorial analogues of their approaches.    
\end{remark}

\section{Strong Lefschetz}

In this section, we deduce Strong Lefschetz from Weak Lefschetz. Again, we work with a generic linear parameter space for $K$ over $k$, a field of characteristic $2$.

\begin{theorem} \label{t:hardlefschetzeven}
    Let $\Sigma$ be a $(d-1)$-dimensional $k$-homology sphere for $d$ even. Then, there exists an element $\ell\in A^1(\Sigma)$ such that for $m\leq d/2$, the multiplication map $\ell^{d-2m}\cdot \colon A^m(\Sigma)\to A^{d-m}(\Sigma)$ is injective.
\end{theorem}

\begin{proof}
    Let $\ell$ be given by Theorem~\ref{t:weaklefschetz}. Let $a\in A^m(\Sigma)$ be nonzero.    
    By iterating weak Lefschetz, $\ell^{d/2-m}a\in A^{d/2}(\Sigma)$ is nonzero, and hence, by anisotropy
    \[0\neq (\ell^{d/2-m}a)^2=a(\ell^{d-2m}a).\]
    Consequently, $\ell^{d-2m}a\neq 0$.    
\end{proof}

\begin{theorem} \label{t:hardlefschetzodd}
    Let $\Sigma$ be a $(d-1)$-dimensional $k$-homology sphere for $d$ odd. Then, there exists an element $\ell\in A^1(\Sigma)$ such that for $m\leq (d-1)/2$, the multiplication map $\ell^{d-2m}\cdot \colon A^m(\Sigma)\to A^{d-m}(\Sigma)$ is injective.
\end{theorem}

\begin{proof}
    Let $S\Sigma$ be the suspension of $\Sigma$. Pick $M_{S\Sigma}$, $M_{\st_{S\Sigma}(v_+)}$, and $M_{\Sigma}$ as in the previous section. Pick $\ell$ as in Theorem~\ref{t:weaklefschetz}. 
    Let $a\in A^m(\Sigma)\cong A^m(\st_{S\Sigma}(v_+))$ be nonzero. Pick $\tilde{a}\in K[\Sigma]^m$ lifting $a$ as in Proposition~\ref{p:weaklefschetzsuspension}. 
    Under
    \[\iota\colon A^e(\Sigma)\cong A^e(\st_{S\Sigma}(v_+))\to A^{e+1}(S\Sigma),\]
    $\iota(\ell^p a)=x_v^{p+1}\tilde{a}$
    for any positive integer $p$.
    Now, by iterating weak Lefschetz, 
    \[\ell^{(d-1)/2-m}a\neq 0,\] and so by Lemma~\ref{l:starlemma},
    \[x_v^{(d+1)/2-m}\tilde{a}=\iota(\ell^{(d-1)/2-m}a)\in A^{(d+1)/2}(S\Sigma)\]
    is  also nonzero.
    Consequently, by anisotropy,
    \[0\neq (x_{v_+}^{(d+1)/2-m}\tilde{a})^2=\tilde{a}(x_{v_+}^{d+1-2m}\tilde{a})
    =\tilde{a}(\iota(\ell^{d-2m}a)).\]
    Therefore, $\ell^{d-2m}a\neq 0$ in $A^*(\Sigma)$.
\end{proof}

\bibliographystyle{plain}
\bibliography{references}

\begin{thebibliography}{10}

\bibitem{Adiprasito:beyondpositivity}
Karim Adiprasito.
\newblock Combinatorial {L}efschetz theorems beyond positivity, 2019.

\bibitem{APP}
Karim Adiprasito, Stavros~Argyrios Papadakis, and Vasiliki Petrotou.
\newblock Anisotropy, biased pairings, and the {L}efschetz property for
  pseudomanifolds and cycles, 2021.

\bibitem{Billera}
Louis~J. Billera.
\newblock The algebra of continuous piecewise polynomials.
\newblock {\em Adv. Math.}, 76(2):170--183, 1989.

\bibitem{BilleraLee}
Louis~J. Billera and Carl~W. Lee.
\newblock A proof of the sufficiency of {M}c{M}ullen's conditions for
  {$f$}-vectors of simplicial convex polytopes.
\newblock {\em J. Combin. Theory Ser. A}, 31(3):237--255, 1981.

\bibitem{Brion}
Michel Brion.
\newblock The structure of the polytope algebra.
\newblock {\em Tohoku Math. J. (2)}, 49(1):1--32, 1997.

\bibitem{BrunsHerzog}
Winfried Bruns and J\"urgen Herzog.
\newblock {\em Cohen-{M}acaulay rings}, volume~39 of {\em Cambridge Studies in
  Advanced Mathematics}.
\newblock Cambridge University Press, Cambridge, 1993.

\bibitem{CLS}
David~A. Cox, John~B. Little, and Henry~K. Schenck.
\newblock {\em Toric varieties}, volume 124 of {\em Graduate Studies in
  Mathematics}.
\newblock American Mathematical Society, Providence, RI, 2011.

\bibitem{Deligne:Weil2}
Pierre Deligne.
\newblock La conjecture de {W}eil. {II}.
\newblock {\em Inst. Hautes \'Etudes Sci. Publ. Math.}, (52):137--252, 1980.

\bibitem{Fulton:toricvarieties}
William Fulton.
\newblock {\em Introduction to toric varieties}, volume 131 of {\em Annals of
  Mathematics Studies}.
\newblock Princeton University Press, Princeton, NJ, 1993.
\newblock The William H. Roever Lectures in Geometry.

\bibitem{Grunbaum}
Branko Gr\"unbaum.
\newblock {\em Convex polytopes}, volume Vol. 16 of {\em Pure and Applied
  Mathematics}.
\newblock Interscience Publishers John Wiley \& Sons, Inc., New York, 1967.
\newblock With the cooperation of Victor Klee, M. A. Perles and G. C. Shephard.

\bibitem{Kalai:many}
Gil Kalai.
\newblock Many triangulated spheres.
\newblock {\em Discrete Comput. Geom.}, 3(1):1--14, 1988.

\bibitem{Karu:HL}
Kalle Karu.
\newblock Hard {L}efschetz theorem for nonrational polytopes.
\newblock {\em Invent. Math.}, 157(2):419--447, 2004.

\bibitem{KLS}
Kalle Karu, Matt Larson, and Alan Stapledon.
\newblock Differential operators, anisotropy, and simplicial spheres, 2025.

\bibitem{KX}
Kalle Karu and Elizabeth Xiao.
\newblock On the anisotropy theorem of {P}apadakis and {P}etrotou.
\newblock {\em Algebr. Comb.}, 6(5):1313--1330, 2023.

\bibitem{Lamotke}
Klaus Lamotke.
\newblock The topology of complex projective varieties after {S}. {L}efschetz.
\newblock {\em Topology}, 20(1):15--51, 1981.

\bibitem{Lazarsfeld1}
Robert Lazarsfeld.
\newblock {\em Positivity in algebraic geometry. {I}}, volume~48 of {\em
  Ergebnisse der Mathematik und ihrer Grenzgebiete. 3. Folge. A Series of
  Modern Surveys in Mathematics [Results in Mathematics and Related Areas. 3rd
  Series. A Series of Modern Surveys in Mathematics]}.
\newblock Springer-Verlag, Berlin, 2004.
\newblock Classical setting: line bundles and linear series.

\bibitem{Lee:stresses}
C.~W. Lee.
\newblock P.{L}.-spheres, convex polytopes, and stress.
\newblock {\em Discrete Comput. Geom.}, 15(4):389--421, 1996.

\bibitem{McMullen}
P.~McMullen.
\newblock The maximum numbers of faces of a convex polytope.
\newblock {\em Mathematika}, 17:179--184, 1970.

\bibitem{McMullen:simple}
Peter McMullen.
\newblock On simple polytopes.
\newblock {\em Invent. Math.}, 113(2):419--444, 1993.

\bibitem{MS:cca}
Ezra Miller and Bernd Sturmfels.
\newblock {\em Combinatorial commutative algebra}, volume 227 of {\em Graduate
  Texts in Mathematics}.
\newblock Springer-Verlag, New York, 2005.

\bibitem{PP}
Stavros~Argyrios Papadakis and Vasiliki Petrotou.
\newblock The characteristic 2 anisotropicity of simplicial spheres, 2020.

\bibitem{Payne:equivariant}
Sam Payne.
\newblock Equivariant {C}how cohomology of toric varieties.
\newblock {\em Math. Res. Lett.}, 13(1):29--41, 2006.

\bibitem{Stanley:UBT}
Richard~P. Stanley.
\newblock The upper bound conjecture and {C}ohen-{M}acaulay rings.
\newblock {\em Studies in Appl. Math.}, 54(2):135--142, 1975.

\bibitem{Stanley:number}
Richard~P. Stanley.
\newblock The number of faces of a simplicial convex polytope.
\newblock {\em Adv. in Math.}, 35(3):236--238, 1980.

\bibitem{Stanley:greenbook}
Richard~P. Stanley.
\newblock {\em Combinatorics and commutative algebra}, volume~41 of {\em
  Progress in Mathematics}.
\newblock Birkh\"auser Boston, Inc., Boston, MA, second edition, 1996.

\bibitem{Swartz}
Ed~Swartz.
\newblock Face enumeration---from spheres to manifolds.
\newblock {\em J. Eur. Math. Soc. (JEMS)}, 11(3):449--485, 2009.

\bibitem{Ziegler}
G\"unter~M. Ziegler.
\newblock {\em Lectures on polytopes}, volume 152 of {\em Graduate Texts in
  Mathematics}.
\newblock Springer-Verlag, New York, 1995.

\end{thebibliography}

\end{document}